\documentclass[12pt,a4paper]{article}
\usepackage[french,english]{babel}
\usepackage[latin1]{inputenc}
\usepackage{hyperref}
\usepackage{amsmath,amssymb,amscd,amsthm,amsfonts}
\usepackage{amsrefs}
\usepackage{mathdots,enumerate}
\usepackage{calrsfs}
\usepackage[all,cmtip]{xy}
\usepackage{rotating}
\usepackage{makeidx}
\usepackage{setspace,graphicx}
\usepackage[french,english]{babel}
\usepackage{color}

\newtheorem{thm}{Theorem}[section]
\newtheorem{defi}[thm]{Definition}

\newtheorem{prop}[thm]{Proposition}
\newtheorem{lemma}[thm]{Lemma}

\newtheorem{rmk}[thm]{Remark}
\DeclareMathAlphabet{\pazocal}{OMS}{zplm}{m}{n}

\newcommand{\A}{\textbf A}

\newcommand{\K}{\textbf K}

\newcommand{\In}{\mathrm In} 
\newcommand{\rr}{\mathcal{R}_\sigma}

\newcommand{\R}{\textbf R}
\newcommand{\x}{\textbf x}
\newcommand{\y}{\textbf y}

\newcommand{\Z}{\textbf Z}
\newcommand{\J}{\mathcal J}
\newcommand{\M}{
\mathcal{M}}
\newcommand{\spec}{\textrm{Spec}\,}
\newcommand{\tr}{\textrm{Trop}\,}
\newcommand{\Supp}{\textrm{Supp}\,}

\providecommand{\keywords}[1]
{
  \textbf{\textit{Keywords--}} #1
}

\providecommand{\subjclass}[1]
{
  \textbf{\textit{2020 Math. Sub. Class--}} #1
}

\title{Groebner fan and embedded resolutions of ideals on toric varieties}
\author{F. Aroca, M. G\'omez-Morales, H. Mourtada }
\date{February 2022}

\begin{document}

\maketitle

\begin{abstract}
\noindent We consider the notions of Groebner fan and Newton non-degeneracy for an ideal on a toric variety, extending the two existing notions for ideals on affine spaces. We prove, without assumptions on the characteristic of the base fields, that the "Groebner fan" of such an ideal is actually
a polyhedral fan and that a sub-variety defined by a Newton non-degenerate ideal on a toric variety $X_\sigma$ admits a toric embedded resolution of singularities $Z\longrightarrow X_\sigma.$ 
\end{abstract}


\section{Introduction}

Resolution of singularities is a fundamental tool in Algebraic Geometry, which also has applications 
in several other domains like Differential Geometry, Analysis and Number Theory. While its existence
when working over fields of characteristic zero was proved by Hironaka in 1968 \cite{Hi}, in positive characteristics this (and its super local couterpart: local uniformization) remains a widely open problem, \textit{e.g.} \cite{CP,CP2,BV,Surf,KFV,T1,Cudim3,CuM,NoS,Sa}.\\

In characteristic zero, a huge amount of work has been done to simplify Hironaka's proof and to make it functorial \cite{V,BM1,EH}. The majority
of this work follows one philosophy which is about making "very" small improvements to the singularity  
by blowing up smooth centers. See also \cite{Mc,ATW} for resolution of singularities using a sequence of weighted blowing ups.\\

In this article, we are concerned with toric resolutions of singularities which, when they exist, are obtained
with one toric morphism. Such resolutions are then very useful in applications and allow to effectively compute
many important invariants of singularities, \textit{e.g.} \cite{Va,DL2,AGL,LMS,RDP,MP,CPPT}. 
For hypersufaces in affine spaces, the so called condition of being Newton non-degenerate, introduced in the seventies by Khovanski, Kuchnirenko and Varchenko \cite{Kou,Va}, guaranties the existence of toric resolution of singularities. Recently, this was generalized to higher codimensions in \cite{Te1,AGS} via the equivalent notions (but expressed in different languages) of Schon varieties and Newton non-degenerate ideals.\\ 

In this article, we extend this last notion to  ideals on toric varieties and prove that their associated varieties admit toric embedded resolutions of singularities
in a sense that will be explained below. Here, it is worth noticing that given a sub-variety $Y$ of a normal toric variety 
$X_\sigma,$ in general there exists no toric embedded resolution $Z\longrightarrow X_\sigma$ of $Y\subset X_\sigma;$
this can already be noticed for curves embedded in the plane \cite{GT,dFGPM,LMR,gen}. But there is a conjecture of Teissier that a toric  resolution exists after embedding in a higher dimensional affine space \cite{T1,gen,hdr,GP1}. We insist here that in the context of this article the embedding is fixed.  We now explain the results of this article; the reader who is unfamiliar with toric geometry
may find a reminder on this subject in section 2.\\

Since the problem is of local nature, we will be considering affine normal toric varieties. Let $\K$ be a perfect field. Let $N$ be a free abelian group of rank $n,$ \textit{i.e.} $N\simeq \Z^n.$  Let $M$ be the dual lattice of $N,$ 
$N=\mathrm{Hom}_{\textbf{Z}}(M,\textbf{Z})$ and let $\sigma \subset N_{\R}=N \otimes_{\Z} \R$ be a strongly convex rational polyhedral cone (\textit{i.e} $,\sigma \cap -\sigma=\{0\}\$).$ Let $\sigma^{\vee} \subset M_{\R}=M \otimes_{\Z} \R$ be the dual cone of $\sigma;$ then the normal toric variety associated with $\sigma$ is by definition  $X_{\sigma}=\spec{\rr}$ where $\rr=\K[\sigma^{\vee}\!\cap M];$ 
note that by definition of $X_\sigma,$ the torus $\textbf{T}=\spec{\K[M]}$ is dense in $X_\sigma.$

Let $f \in \rr;$ hence we can write 
$$f=\sum_{\alpha \in  \sigma^{\vee}\!\cap M}a_{\alpha}\x^{\alpha},$$
where $\x^{\alpha}=x_1^{\alpha_1}\cdots x_n^{\alpha_n}$. Since $\sigma$ (and hence $\sigma^\vee$) is strongly convex, the expression of $f$ in this form is unique. Note that the support of $f,$ $\mathrm{Supp}(f):=\{\alpha \in \sigma^{\vee}\!\cap M ; a_{\alpha}\not=0\}$ is a finite set. We define the valuation associated with $v\in N_\R$, $\nu_v:\rr \longrightarrow \Z,$ by  
$$ \nu_v(f)=\mathrm{min}\{ v \cdot \alpha~; \alpha \in
\mathrm{Supp}(f) \}~~~~\mbox{for any}~~f\in\rr \setminus \{0\}.$$
Here, $v \cdot \alpha$ stands for the action of $v\in N_\R$ on $\alpha \in M$ being defined by the scalar product; in which we set $\nu_v(0):=\infty.$ The initial form of $f\not=0$ with respect to $v$ is then defined by
$$\In_v(f)=\sum_{\{\alpha \in \mathrm{Supp}(f);v \cdot \alpha= \nu_v(f) \}}a_\alpha \x^{\alpha},$$ and the initial form of $0$ is defined as $0$. Both objects $\nu_v(f)$ and $\In_v(f)$ are well defined, thanks to the uniqueness of the expression of $f$ in $\K[\sigma^{\vee}\!\cap M],$ as mentioned above. Let $\J \subset R_\sigma$ be an ideal. For $v \in \sigma,$ we define the initial ideal $\In_v(\J)$ of $\J$  with respect to  $v$ as follows:
$$\In_v(\J)=\{\In_v(f);f \in \J\}. $$
Notice that $\In_v(\J)$ is an ideal. The ideal $\J$ defines an equivalence relation $\sim$ on the cone $\sigma$ as follows: let $v_1,v_2 \in \sigma;$ then we set

$$v_1\sim v_2~~~\mbox{if}~~~ \In_{v_1}(\J)=\In_{v_2}(\J).$$
We prove the following (see \ref{Groebner}):
\begin{thm}\label{GB} The closures (with respect to the topology of $\R^n$) of the equivalence classes of the relation $\sim$ defines a polyhedral fan which is a subdivision of $\sigma.$

\end{thm}

The fan in theorem \ref{GB} is called the Groebner fan; it generalizes the well known Groebner fan of ideals in polynomial rings cite \cite{MaRo,As,FJT,BaTa}. Using \cite{FJT}, we actually have an algorithm to compute
this fan.\\

We will use this notion of Groebner fan to produce embedded resolutions of singularities of sub-varieties of toric varieties defined by ``non-degenerate ideals" which are defined as follows:

\begin{defi}\label{nndt}
Let $\J\subset \rr$ be an ideal. We assume that the closed orbit $O$ (which correspond to the cone $\sigma$) belongs to $V(\J).$  We say that $\J$ is Newton non-degenerate at $O$ if, for every $v\in \sigma,$ 
we have $$\mathrm{Sing}(V(\In_v(\J))\cap \textbf{T}=\emptyset;$$

\noindent here $\mathrm{Sing}(V(\In_v(\J))$ denotes the singular locus of the sub-variety of $X_\sigma$ defined by the ideal  $\In_v(\J).$

\end{defi}

Another main result of this article is that a non-degenerate ideal admits a toric embedded resolution in the following sense.

\begin{thm}\label{Main} Let $\J\subset \rr$ be a Newton non-degenerate  ideal and let $Y=V(\J)\subset X_\sigma.$ Let $\Sigma$ be a regular subdivision of $\sigma$ which is compatible with the Groebner fan of $\mathcal{J}.$
Then the associated toric morphism $\pi_\Sigma:X_\Sigma\longrightarrow X_\sigma$ is a proper birational morphism, in which irreducible components of the total transform $\pi_\Sigma^{-1}(Y)$ are smooth (in particular, the strict transform of $Y$ is smooth) and meet transversely.

\end{thm}
Note that the variety $X_\Sigma$ is smooth since it is associated with a regular fan. The main result of \cite{AGS} is the special case of theorem \ref{Main}, where $\sigma$ is considered to be $\R_+^n.$  In the special case where $Y$ is a hypersurface, a version of theorem \ref{Main} was proved in \cite{St}, see also \cite{NS}. The full generality of theorem \ref{Main} is used in \cite{MS2} (see also \cite{MS1}).

\begin{rmk} The embedded resolution of singularities in theorem \ref{Main} should be understood as an embedded resolution of singularities of the pair $Y\subset X_\sigma$ in a small neighbourhood of $O.$ 

\end{rmk}

\textbf{Acknowledgements.} Hussein Mourtada would like to thank the UNAM, Mexico City, for its hospitality during the preparation of this article. 

\section{The toric context}

We refer to \cite{Fu,Cox} for the basics on normal toric varieties. As in the intro-duction  $\K$ is a perfect field; $N$ is a free abelian group of rank $n,$ \textit{i.e.} $N\simeq \Z^n;$ $M\simeq \Z^n$ is the dual lattice of $N,$ 
$N=\mathrm{Hom}_{\textbf{Z}}(M,\textbf{Z})$ and $\sigma \subset N_{\R}=N \otimes_{\Z} \R\simeq \R^n$ is a strongly convex rational polyhedral cone (of dimension $n$) which is generated by $v^1,\ldots,v^r$, \textit{i.e.},
$$\sigma=\{\lambda_1v^1+\cdots+\lambda_rv^r; \lambda_i\in \R_{\geq 0}~~\mbox{for}~~i=1,\ldots,r \}.$$
We denote by $\sigma^{\vee} \subset M_{\R}=M \otimes_{\Z} \R\simeq \R^n$ the dual cone of $\sigma;$  the normal toric variety associated with $\sigma$ is by definition  $X_{\sigma}=\spec{\rr}$ where $\rr=\K[\sigma^{\vee}\!\cap M].$ Let $\alpha^{1},\ldots,\alpha^s$ be a minimal set of generators of $\sigma^{\vee}\!\cap M$ as a semigroup.  If we write 
$\K[M]=\K[x_1^{\pm 1},\ldots,x_n^{\pm 1}],$
then we have $$\rr=\K[\x^{\alpha^1},\ldots,\x^{\alpha^s}];$$ here we use the notation $\x^{\alpha^i}:=x_1^{\alpha^i_1}\cdots x_n^{\alpha^i_n}.$ This allows to embed $X_\sigma \hookrightarrow \A^s$ as follows: let $$\Psi:\K[y_1,\ldots,y_s]\longrightarrow \rr=\K[\x^{\alpha^1},\ldots,\x^{\alpha^s}]$$
be the  $\K-$algebra morphism defined by $\Psi(y_i)=\x^{\alpha^i}.$ Let 
$\mathcal{M}$ be the $n\times s$ matrix whose columns are the $\alpha^i\,$'s, \textit{i.e.} $(\M)_{k,l}=\alpha^l_k.$ For 
$\gamma \in \Z_{\geq 0}^s,$ we use the notation $\y^\gamma:=y_1^{\gamma_1}\cdots y_s^{\gamma_s};$ an element $h \in \K[\y]:=\K[y_1,\ldots,y_s]$ is expressed as follows:
$$h(\y)=\sum_{\gamma \in \Z_{\geq 0}^s}a_\gamma \y^{\gamma}=\sum_{\gamma \in \Supp(h)}a_\gamma \y^{\gamma},$$
where $a_\gamma \in \K$ is nonzero for a finite number of $\gamma.$
The morphism $\Psi,$ being a $\K$-algebra morphism, satisfies 
\begin{equation}\label{Psi}
\Psi (h)=\sum_{\gamma \in \Supp(h)}a_\gamma \x^{\M\gamma}.
\end{equation}

Notice that the morphism $\Psi$ is surjective
and hence the ideal $I_\sigma:=\mathrm{Ker}(\Psi)$ is the defining ideal of $X_\sigma$ in $\A^s.$ The affine space $\A^s$ has the structure of a toric variety which is associated with the cone $\R_+^s.$ We embed $\sigma$ in $\R_+^s$ via the following linear map:

$$\phi:\R^n\longrightarrow \textbf{R}^s,$$
$$v \mapsto \phi(v)=(v \cdot \alpha^1,\ldots, v \cdot \alpha^s).$$
The linear map $\phi$ is injective; indeed, since the vector space $\R^n$ is generated by $\alpha^1,\ldots,\alpha^s$ (the cone $\sigma$ being of full 
dimension and hence so is $\sigma^{\vee}$), the kernel of $\phi$ is reduced to $\{0\}.$ Note that since 
the $\alpha^i\,$'s are in $\sigma^{\vee},$ we have 
$$\phi(\sigma)\subset \R_+^s.$$

For $\omega \in \R_+^s, h \in \K[\y]$ and $J\subset \K[\y]$ an ideal, the objects $\nu_\omega(h), \In_\omega(h)$
and $\In_\omega(J)$ are definied in the same way as in the introduction; this is the special case $\sigma=\R_+^s.$\\

In the sequel, we will use the notation $v$ for a vector in a given $\sigma$ and $\omega$ for a vector in $\R_+^s;$
we will use the notation $\J$ for an ideal in $\rr$ and $J$ for an ideal in $\K[\y].$

\section{Proofs of the main results}

We keep the notations of the precedent sections. Let $v \in \sigma$ and set $\omega=\phi(v).$ Let $\J$ be an ideal of $\rr;$ then $J:=\Psi^{-1}(\J)$ is an  ideal of $\K[\y].$ In order to determine a relation between $\In_v(\J)$ and $\In_\omega(J),$
we begin by comparing, for $h \in \K[\y]$ and $f=\Psi(h);$ $\nu_w(h)$ with $\nu_v(f)$ and $\In_\omega(h)$ with $In_{v}(f).$

\begin{lemma}\label{valuation} Let $v \in \sigma$ and set $\omega=\phi(v).$ Let $h \in \K[\y]$ and $f=\Psi(h).$ We have 
$$\nu_\omega(h)\leq \nu_v(f).$$
Moreover, the inequality is strict if and only if $\Psi(\In_\omega(h))=0.$
\end{lemma}
\begin{proof} For a monomial $\y^\gamma=y_1^{\gamma_1}\cdots y_s^{\gamma_s},$
notice that

\begin{align}\nonumber
\nu_\omega(\y^\gamma)&=\omega \cdot\gamma=\sum_{i=1}^s\omega_i\gamma_i=\sum_{i=1}^s \left(\sum_{j=1}^nv_j\alpha^i_j\right)\gamma_i=\sum_{j=1}^nv_j\left(\sum_{i=1}^s\alpha^i_j\gamma_i\right)\\\label{mon} 
&=v\cdot (\M \gamma)=\nu_v(\x^{\M\gamma})=\nu_v(\Psi(\y^\gamma)).
\end{align}

\noindent If we write $$h=\sum_{\{\gamma \in \Supp(h);~~\omega \cdot \gamma = \nu_\omega(h)\}}a_\gamma \y^{\gamma}+\sum_{\{\gamma \in \Supp(h);~~\omega \cdot \gamma > \nu_\omega(h)\}}a_\gamma \y^{\gamma},$$
it follows from the formulas (\ref{Psi}) and (\ref{mon}) that 
$$f=\sum_{\{\gamma \in \Supp(h);~~v \cdot (\M\gamma) = \nu_\omega(h)\}}a_\gamma \x^{\M\gamma}+\sum_{\{\gamma \in \Supp(h);~~v \cdot (\M\gamma) > \nu_\omega(h)\}}a_\gamma \x^{\M\gamma}.$$

In particular, we have that $\nu_v(f)=\nu_w(h)$ if 
\begin{equation}\label{anh}
\sum_{\{\gamma \in \Supp(h);v \cdot \M\gamma = \nu_\omega(h)\}}a_\gamma \x^{\M\gamma}=\Psi(\In_\omega(h))\not=0,
\end{equation}

while the strict inequality $\nu_v(f)>\nu_w(h)$ holds if not; note that the equality in (\ref{anh}) may happen because the linear morphism defined by $\M$ is not injective in general.
Moreover, if $\nu_v(f)=\nu_w(h),$ we have 
\begin{equation}\label{Initials}
\Psi(\In_\omega(h))=\In_v(f)
\end{equation}
\end{proof}

Recall that the local tropical variety \cite{PPS} (see also \cite{trop}) at the origin ($O$ of $\A^s$) of an ideal $I\subset \K[\y]$
verifying $O \in V(I)$ is given by $$\tr(I):=\{\omega \in \R_+^s~~\mbox{such that}~~\In_\omega{I}~~ \mbox{does not contain monomials}\}.$$

We determine in the following proposition the local tropical variety of $I_\sigma \subset \K[\y];$ a result which is also of independent interest. We prove it here because it makes use of formula (\ref{mon}) that we have just showed.
\begin{prop}\label{torictrop}

\begin{enumerate}

\item Let $\omega \in \mbox{Trop(I)}.$ We have 
$$\In_\omega(I_\sigma)=I_\sigma.$$

\item We have $$ \phi(\sigma)= \tr(I_\sigma).$$

\end{enumerate}

\begin{proof}
\begin{enumerate}
\item By the main result of \cite{KVA}, to determine $\In_\omega(I_\sigma),$ it is sufficient to determine a Groebner basis of $I_\sigma \subset \K[\y]$ with respect to a total monomial order which refines the preorder (\textit{i.e.} partial order) defined by $\omega.$ Now, it is direct to see that a reduced Groebner basis of
a binomial ideal (which is the case for $I_\sigma$) is a binomial ideal \cite{ES}: this follows from the fact that the $S-$polynomial of two binomials is again a binomial. So, let $(h_1,\ldots,h_l)=I_\sigma$ be such a Groebner basis of $I_\sigma;$ then, the $h_i\,$'s are binomials  and $\In_\omega(I_\sigma)=\left(\In_\omega(h_1),\ldots,\In_\omega(h_l)\right)$ by \cite{KVA}. But since $\omega \in \mbox{Trop(I)},$ $\In_\omega(h_i)$ cannot be a monomial for $i=1,\ldots,l$; hence
we have $\In_\omega(h_i)=h_i,$ for $i=1,\ldots,l,$ since this latter is a binomial. We deduce that $\In_\omega(I_\sigma)=I_\sigma.$

\item It follows from formula (\ref{mon}) that $$\phi(\sigma)\subset \mbox{Trop}(I_\sigma),$$
where, recall that $I_\sigma=Ker(\Psi)$ is the defining ideal of $X_\sigma$ in $\A^s.$ Indeed, 
if $\omega=\phi(v) \in \phi(\sigma),$ since the $\alpha^i\,$'s are in $\sigma^{\vee},$ the components of $\omega$
are positive or zero. Moreover, a binomial $\y^\gamma - \y^{\gamma'} \in I_\sigma$ if and only if 
$\Psi(\y^\gamma)=\Psi(\y^{\gamma'}).$ 
Hence, by formula (\ref{mon}), we have 
$$\nu_\omega(\y^\gamma)=\nu_v(\Psi(\y^\gamma))=\nu_v\left(\Psi(\y^{\gamma'})\right)=\nu_\omega(\y^{\gamma'}); $$
we deduce that $\omega \in \mbox{Trop}(I_\sigma).$ Now, by what we have just said, $\omega$ belongs to $\mbox{Trop}(I_\sigma)$ if $\omega \in \R_+^s$ and $\omega \cdot \gamma = \omega \cdot \gamma' $ for every $\gamma$ and $\gamma'$
satisfying $\M\gamma=\M \gamma'.$ This implies that $\mbox{Trop}(I_\sigma)$ is the intersection of 
the orthogonal $(Ker(\M))^\perp$ of $Ker(\M)$ with $\R_+^s.$ Since the cone $\sigma^\vee$ is of maximal dimension,
we know that the vector space generated by the $\alpha^i\,$'s is of dimension $n;$ hence $\M$ is of rank $n,$ therefore $Ker(\M)$ is 
of dimension $s-n$ and $(Ker(\M))^\perp$ is of dimension $n.$ So, because we have proved that
$\phi(\sigma)\subset \mbox{Trop}(I_\sigma) \subset (Ker(\M))^\perp,$ the vector space generated by $\phi(\sigma)$
and $(Ker(\M))^\perp$ are equal; The intersection of this vector space with  $\R_+^s$ is
$\phi(\sigma)\subset \mbox{Trop}(I_\sigma).$

\end{enumerate}

\end{proof}

\end{prop}

\begin{prop}\label{maximum} Let $v \in \sigma$ and set $\omega=\phi(v).$ Let $f\in \rr \setminus \{0\}.$
We have that $$\mbox{Max}\{\nu_\omega(h); \,h\in \Psi^{-1}(f) \}$$ is finite and for  $h\in \Psi^{-1}(f)$ realizing this finite maximum we have
$$\nu_{\omega}(h)=\nu_v(f) ~~~\mbox{and}~~~ \Psi(\In_\omega(h))=\In_v(f).$$

\end{prop}
\begin{proof}
By lemma (\ref{valuation}), we have the fact that 
$$\mbox{Max}\{\nu_\omega(h);~~h\in \Psi^{-1}(f) \} \,\leq\; \nu_v(f);$$
hence this maximum is finite because $\nu_v(f)<\infty,$ for $f$ nonzero.
Let $g\in \Psi^{-1}(f)$ be such that $\nu_\omega(g)=\mbox{Max}\{\nu_\omega(h);~~h\in \Psi^{-1}(f) \}.$
Suppose (by contradiction) that $\nu_\omega(g)<\nu_v(f);$ by lemma (\ref{valuation}), this implies that $\Psi(\In_v(g))=0;$
hence $\Psi(g-\In_v(g))=f.$ At the same time we have 
$$ \nu_\omega(g-\In_\omega(g))>\nu_\omega(g);$$
this contradicts the fact that $\nu_\omega(g)=\mbox{Max}\{\nu_\omega(h);~~h\in \Psi^{-1}(f) \}.$ Hence
we have $\nu_\omega(g)=\nu_v(f);$ by (\ref{Initials}) this gives $\Psi(\In_\omega(g))=\In_v(f).$

\end{proof}

Let $\J$ be an ideal of $\rr;$ then $J:=\Psi^{-1}(\J)$ is an  ideal of $\K[\y].$ We are now ready to compare $\In_v(\J)$ and $\In_\omega(J)$ and deduce the structure of the Groebner fan of $\J$ from the Greobner fan of $J.$

\begin{prop}\label{ideals} Let $\J$ be an ideal of $\rr$ and let $J:=\Psi^{-1}(\J)$ be an ideal of $\K[\y].$ Let $v \in \sigma$ and let $\omega=\phi(v).$ We have

$$\In_v(\J)=\Psi(\In_\omega(J)).$$

\end{prop}
 \begin{proof}
 Let $F\in \In_v(\J),F\not=0;$ there exists $f\in \J$ such that $F=\In_v(f).$ Let $g\in \Psi^{-1}(f)$ be such that $\nu_\omega(g)=\mbox{Max}\{\nu_\omega(h);~~h\in \Psi^{-1}(f) \}.$ By proposition \ref{maximum}, we have $F=\Psi(\In_\omega(g));$ hence $F\in \Psi(\In_\omega(J))$ and we have the inclusion $\In_v(\J)\subset\Psi(\In_\omega(J)).$ \\
 
 Let $G\in \In_\omega(J);$ then $G=\In_\omega(g), g\in J.$ By lemma \ref{valuation}, we have either $\Psi(G)=0 \in \In_v(\J)$, or we have $\nu_\omega(g)=\nu_v(\Psi(g));$ and thus by (\ref{Initials}) this gives  $\Psi(G)=\In_v(\Psi(g))\in \In_v(\J).$ We deduce the reverse inclusion and the proposition.
\end{proof}

One of the two main results of this paper is the following:
\begin{thm}\label{Groebner} Let $v,v' \in \sigma$ and let $\omega=\phi(v),~\omega'=\phi(v').$ Let $\J \in \rr$ and let $J:=\Psi^{-1}(\J).$ We have 
\begin{equation}\label{fan}
\In_v(\J)=\In_{v'}(\J)~~\mbox{if and only if}~~ \In_\omega(J)=\In_{\omega'}(J).
\end{equation}

In particular the equivalence classes of the relation $\sim$ defines a polyhedral fan which is a subdivision of $\sigma$ and which is obtained from the Groebner fan
of $J$ intersected with $\sigma.$
\end{thm}

\begin{proof} The reverse implication of $(\ref{fan})$ is a direct consequence of proposition \ref{ideals}. Let us prove the direct implication. For that, we suppose that $\In_\omega(J)\not=\In_{\omega'}(J).$ Since, $I_\sigma=Ker(\Psi) \subset J$, we have by proposition \ref{torictrop} (parts 1 and 2)  
$$ \In_\omega(I_\sigma)=I_\sigma \subset \In_\omega(J) ~~\mbox{and}~~\In_{\omega'}(I_\sigma)=I_\sigma \subset \In_{\omega'}(J).$$
So we can choose basis of the forms
$$ \In_\omega(I_\sigma)\subset \In_\omega(J)=(G_1,\ldots,G_l,H_1,\ldots,H_k)$$
$$ \In_{\omega'}(I_\sigma)\subset \In_{\omega'}(J)=(G'_1,\ldots,G'_{l'},H_1,\ldots,H_k),$$
where $I_\sigma=(H_1,\ldots,H_k).$ 
%Moreover, such basis can be obtained (by applying the main result of $\cite{KVA},$ see proposition \ref{torictrop}) by considering a Groebner basis with respect to monomial orders refining the preorders defined by $\omega$ and $\omega';$ we can assume that the for the refined  monomial orders, the leading monomials of the $H_i$'s are the same and by reducing we can assume that these leading monomials of the $H_i$'s do not divide any monomial appearing neither in the $G_i$ 's  nor in the $G_i^'$ 's. This implies that there exist no couple $(i,j) \in \{1,\ldots, l\}\times \{1,\ldots, l'\}$ such that $\Psi(G_i)=\Psi(G'_j).$ Indeed, if so we will have that 

Now, since we assumed that $\In_\omega(J)\not=\In_{\omega'}(J),$ there exist an $i$ such that $G_i \not\in \In_{\omega'}(J).$
We have $\Psi(G_i)\not\in \Psi(\In_{\omega'}(J)):$ indeed, otherwise we will have $\Psi(G_i-G')=0$ for some $G'\in\In_{\omega'}(J);$ this implies that $G_i-G'=H\in I_\sigma$ and hence $G_i=G'+H \in \In_{\omega'}(J),$ which is a contradiction. We deduce, since $\Psi(G_i)\not\in \Psi(\In_{\omega'}(J))$, that  $\In_v(\J)\not=\In_{v'}(\J).$ So we have proved that
if $\In_\omega(J))\not=\In_{\omega'}(J))$ then $\In_v(\J)\not=\In_{v'}(\J);$ this ends the proof.

\end{proof}

\begin{rmk} We keep the notation of theorem \ref{Groebner}. Using the main result of \cite{FJT}, we know how to compute the Groebner fan of $J$ (\textit{i.e.} there is an algorithm for this); hence, applying theorem \ref{Groebner}, we obtain an algorithm to compute the (toric) Groebner fan of $\J.$ At the same time, in practice and for dimension reasons, it should be much easier to compute directly the (toric) Groebner fan of $\J$; indeed, in general $s$ is much greater than $n.$  
\end{rmk}

We now are ready to prove theorem \ref{Main}.

\begin{proof}[Proof of theorem \ref{Main}.]
For $\J \in \rr$ satisfying the Newton non-degeneracy property, we will prove that $J:=\Psi^{-1}(\J)\subset \K[\y]$ is Newton non-degenerate. For that, we need to prove that for any $\omega \in \R_+^s,$ the singular locus of $V(\In_\omega(J))\subset V((y_1\cdots y_s)).$
Thanks to proposition \ref{torictrop}, we only need to consider the case where $\omega \in \phi(\sigma)$; otherwise, for $\omega \not\in \phi(\sigma)=\tr(I_\sigma),$ since $I_\sigma \subset J,$ we have that  $\In_\omega(J)$ contains monomials and hence $V(\In_\omega(J))\subset V((y_1\cdots y_s)),$ in particular its singular locus is. Now, assume that $\omega \in \phi(\sigma);$ by proposition \ref{ideals}, we have that $\Psi(\In_\omega(J))=\In_v(\J);$ hence we have $$ V(\In_v(\J))=V(\In_\omega(J))\subset X_\sigma \subset \A^s.$$ We hence know, by the Newton non-degeneracy condition defined in \ref{nndt} that the singular locus is included in the complement of the torus $(\K^*)^n$ in $X_\sigma.$   So we need to prove that the complement of the torus $(\K^*)^n$ in $X_\sigma$ is sent by the embedding $X_\sigma\subset \A^s$ to the complement of the torus $(\K^*)^s.$ Note that the embedding $(\K^*)^n \hookrightarrow (\K^*)^s$ is given by

$$(t_1,\ldots,t_n) \mapsto (\textbf{t}^{\alpha^1},\ldots,\textbf{t}^{\alpha^s}),$$
where $\textbf{t}^{\alpha^i}=t_1^{\alpha^i_1}\cdots t_n^{\alpha^i_n}.$ Now, we know (\textit{e.g.} section 3.2 in \cite{Cox}) that 
$$ X_\sigma \setminus (\K^*)^n=\bigsqcup_{0 \not=\tau \subseteq \sigma} \mathbb{O}_\tau,$$ where $\mathbb{O}_\tau$ is 
the orbit associated to $\tau;$ note that the orbit $\mathbb{O}_0$ is equal to $(\K^*)^n.$ Then, by the definition of $\mathbb{O}_\tau$ we know that if $\alpha^i \not\in \tau^{\perp},$ any point $p$ which is in $\mathbb{O}_\tau$  will belong to $V(y_i)\subset \A^s.$ Hence, we deduce that $J$ is Newton non-degenerate as an ideal in  $\K[\y].$  
By the main theorem of \cite{AGS}, we know that any regular subdivision $\widetilde{\Sigma}$ of the Groebner fan of $J$ induces
a proper birational map $Z_{\widetilde{\Sigma}}\longrightarrow \A^s$ (where $Z_{\widetilde{\Sigma}}$ is the toric variety associated 
with the fan $\widetilde{\Sigma}$), which is an embedded resolution of $V(Y)\subset \K^s.$ Now, by theorem \ref{Groebner}, the fan $\tilde{\Sigma}$ induces a 
regular subdivision of the (toric) Greobner fan of $\J$ (actually, one can go the other way and construct $\tilde{\Sigma}$ from a regular subdivision  of the -toric- Groebner fan of $\J).$ The morphism $Z_{\widetilde{\Sigma}}\longrightarrow \A^s$ is an embedded resolution of $X_\sigma \subset \A^s $ (see  \cite{GPT}) and the strict transform of $X_\sigma$ is $X_\Sigma.$ The strict transform of $V(\J)\subset X_\sigma $ in $X_{\Sigma}$ is the same
as the strict transform $Y'$ of $V(J)\subset \A^s$ in $Z_{\widetilde{\Sigma}}.$ In particular, this strict transform is smooth, and since it is transverse to the irreducible components of the exceptional locus of $Z_{\widetilde{\Sigma}}\longrightarrow \A^s$, we deduce that $Y'$ is transverse to the exceptional locus of $X_\Sigma \longrightarrow X_\sigma$, which is nothing but the transverse intersection of the exceptional locus of $Z_{\widetilde{\Sigma}}$ with $X_\Sigma$ and the theorem follows.

\end{proof}
\bibliographystyle{amsrefs}
\bibliography{}

[Fuensanta Aroca] Universidad Nacional Aut\'onoma de M\'exico, Unidad Cuernavaca, A.P. 273-3, C.P. 62251, Cuernavaca,
MOR, Mexico.\\
fuen@matcuer.unam.mx\\

[Mirna G\'omez-Morales] Universidad Nacional Aut\'onoma de M\'exico, Unidad Cuernavaca, A.P. 273-3, C.P. 62251, Cuernavaca,
MOR, Mexico.\\
fm.gomez@im.unam.mx\\

[Hussein Mourtada] Universit\'e de Paris, Sorbonne Universit\'e, CNRS, Institut de Math\'ematiques de Jussieu-
Paris Rive Gauche, F-75013 Paris, France.\\
hussein.mourtada@imj-prg.fr

\end{document}